\documentclass[12pt,reqno]{amsart}
\usepackage{amsthm,amssymb,mathabx}
\usepackage{bm,fullpage}

\usepackage[
  bookmarks = true, 
  pagebackref = false, 
  pdfauthor = {Anderson Cook Kumchev Hughes}, 
  pdftitle = \text{ACHK : Ergodic Waring--Goldbach}, 
  colorlinks = true, 
  linkcolor = blue, 
  citecolor = blue]
{hyperref}

\newtheorem{thm}{Theorem}

\newtheorem{lemma}{Lemma}

\theoremstyle{remark}

\newtheorem{remark}{Remark}

\newcommand{\R}{\ensuremath{\mathbb{R}}}
\newcommand{\T}{\ensuremath{\mathbb{T}}}
\newcommand{\Z}{\ensuremath{\mathbb{Z}}}

\newcommand{\N}{\ensuremath{\mathbb{N}}}

\newcommand{\bfa}{\ensuremath{\mathbf{a}}}
\newcommand{\bfq}{\ensuremath{\mathbf{q}}}

\newcommand{\la}{\ensuremath{\lambda}}

\newcommand{\eq}{\begin{equation}}
\newcommand{\ee}{\end{equation}}

\newcommand{\form}{\mathfrak{f}}
\newcommand{\dimension}{n}

\newcommand{\theprimes}{\mathbb{P}}

\DeclareMathOperator{\lcm}{lcm}
\newcommand{\Gammakn}[1][{n,k}]{\Gamma_{#1}}

\newcommand{\singseries}{\mathfrak S(\lambda; \mathbf a, \mathbf q)}

\numberwithin{equation}{section}
\title{Quantitative $l^p$-improving for discrete spherical averages along the primes}
\author[T. Anderson]{Theresa C. Anderson}
\address{
	Department of Mathematics
	\\ Purdue University
	\\ 150 N. University St.
	\\ W. Lafayette, IN 47907
	\\	U.S.A.
}
\email{tcanderson@purdue.edu}

\begin{document}
\maketitle

\begin{abstract}
We show quantitative (in terms of the radius) $l^p$-improving estimates for the discrete spherical averages along the primes.  These averaging operators were defined in \cite{ACHK} and are discrete, prime variants of Stein's spherical averages.  The proof uses a precise decomposition of the Fourier multiplier.  
\end{abstract}

\section{Introduction}

The search for $L^p(\R^n)$-improving capabilities of operators in harmonic analysis is a classic line of investigation.  It is interesting in many cases to see how much the operator ``improves" the the $L^p$ norm (by achieving a higher value of $p$).  In particular, determining these bounds for operators involving integration over a curved submanifold is an active research area (see the early work of Littman \cite{Littman} and Strichartz \cite{Str}).  With discrete operators, that is operators defined over the integer lattice, $l^p(\Z^n)$ spaces behave differently; due to the nesting properties, ``improving" seems to become a trivial consequence of $l^p$-inequalities.  However this is not completely the case, as nontrivial quantitative improving estimates can shed light on the behavior of these operators.    

In this paper, we consider the discrete spherical averaging operator along the primes, first developed (to the best of our knowledge) in \cite{ACHK}.  This is a discrete, prime variant of Stein's spherical averaging operator, and number theoretic properties therefore come into play in its analysis.  For more information and history, see \cite{ACHK}.  Inspired by recent interest in $l^p$-improving for the integer version of this operator in \cite{Hughes}, \cite{KL}, \cite{K1}, we prove quantitative $l^p(\Z^n)$ improving estimates for the discrete spherical averages along the primes in terms of the radius $\lambda$.  These take the form of $l^p(\Z^n) \to l^{p'}(\Z^n)$ bounds for any $1\leq p \leq 2$, and the decay rate improves as $p$ approaches 1.  

To state our main theorems, we require a few definitions.  Discrete spherical averages were introduced by Magyar in \cite{Magyar discrete}:
\[ S_\lambda f(\mathbf x) := \frac 1{\#\{ \mathbf y \in \mathbb Z^n : |\mathbf y|^2 = \lambda \} } \sum_{|\mathbf y|^2 = \lambda} f(\mathbf x - \mathbf y) \]
where $|y|^2 = y_1^2 + \dots y_n^2$.
These played a role in Magyar--Stein--Wainger's proof of sharp $l^p(\Z^n)$ bounds for the corresponding maximal function \cite{MSW}
\[ S_*f (\mathbf x) := \sup_{\lambda \in \mathbb N} |S_\lambda f(\mathbf x) |. \]
Note that we have chosen to define these averages along spheres of radius $\lambda^{1/2}$.

Recently both Hughes \cite{Hughes} and Kesler--Lacey \cite{KL} used Magyar and Magyar--Stein--Wainger's techniques to find $l^p$-improving estimates for these spherical averages, with decay in terms of $\lambda$, showing that for dimensions 5 and greater, and $\frac{n}{n-2} \leq p \leq 2$, there exists constants independent of $\lambda$, such that for all $\lambda \in \N$,
\[
\|S_\lambda f\|_{l^{p'}(\Z^n)} \lesssim\lambda^{-\eta_p}\|f\|_{l^p(\Z^n)}
\]
where $\eta_p = \frac{n}{2}(\frac{2}{p}-1)$.  (At first glance Hughes's and Kesler--Lacey's decay rate might look different, but we remind the reader that Kesler--Lacey define their averages along spheres of radius $\lambda$, not $\lambda^{1/2}$, so these decay rates are the same.)    This decay rate is also optimal in this range of $p$.  At publication time, these ranges obtained have been updated to the wider $\frac{n-1}{n+1} \leq p \leq 2$ in \cite{KL} and $\frac{n+1}{n+1} \leq p \leq 2$ in \cite{Hughes} (with an $\varepsilon$ loss).  Please see both of these papers, as well as the related works \cite{K1}, \cite{K2}, \cite{KLM}, for more discussion of these bounds, as well as other related results.  

The discrete averaging operators we consider are averages over the prime vectors (or 'prime points') on the algebraic surface
\begin{equation}
\label{sphere}
\form(\mathbf x) := |\mathbf x |^k := x_1^k+ \dots + x_n^k = \lambda, 
\end{equation} 
The Waring--Goldbach problem in analytic number theory involves the study of these points.  Classic work of Hua \cite{Hua_book} established an asymptotic for the number of these points; as long as we restrict to a specific arithmetic progression $\Gamma_{n,k}$, 
\begin{equation}
\label{Hua}
P(\lambda) \sim \mathfrak S_{n,k}(\lambda) \lambda^{{n/k - 1}},
\end{equation} 
where $\mathfrak S_{n,k}$ is a singular series that for our purposes can be regarded as a constant, and $P(\lambda)$ denote the number of prime solutions of \eqref{sphere},
\[ P(\lambda) = \sum_{\form(\mathbf p) = \lambda} \log\mathbf p, \] counted with logarithmic weights where $\log\mathbf x = (\log x_1) \cdots (\log x_n)$ and ${\bf p}$ is a vector with all coordinates prime.  It is natural to count these prime lattice points with density due to the prime number theorem.  

The authors in \cite{ACHK} study an ergodic version of this problem by asking quantitative distributional questions about these points, and answer these by proving $l^p(\Z^n)$ bounds for the discrete spherical maximal function along the primes.  We consider these discrete spherical averages defined for ${\bf l}\in\Z^n$ 
\begin{equation}\label{primeavg} 
A_\lambda f({\bf l}) := \sigma_\lambda \star f({\bf l}) = \frac{1}{P(\lambda)} \sum_{|{\bf p}|^k=\lambda} (\log{\bf p}) f({\bf p-l}). 
\end{equation}
where $\sigma_\lambda$ is a probability measure on $\Z^n$ defined by
\[ \sigma_\lambda({\mathbf x}) := \frac{1}{P(\lambda)}{\bf 1}_{\{{\mathbf p} \in \theprimes^\dimension: |\mathbf p|^k=\lambda\}}({\mathbf x}) \log {\mathbf x}, \]
and $\lambda \in \Gamma_{n,k}$, an infinite arithmetic progression of radii where the Hua asymptotic \ref{Hua} holds.
Let $1 \leq p \leq 2$ and define $n_0(k) = 2^k+1$ when $k = 2, 3$ or $4$, and
\[ n_0(k) = k^2 + 3 - \max_{1 \le j \le k-1} \left\lceil \frac {kj - \min(2^j,j^2+j)}{k-j+1} \right\rceil \]
when $k \ge 5$.  Let $\lambda \in \Gamma_{n,k}$.  We now recall one of the main results in \cite{ACHK} as we will use part of this decomposition.  For more details, background, and other subtleties of the ergodic Waring-Goldbach problem, see \cite{ACHK}.

 For an integer $q \ge 1$, we write $\mathbb Z_q = \mathbb Z/q\mathbb Z$ and $\mathbb U_q = \mathbb Z_q^*$, the group of units. If $\mathbf q = (q_1, \dots, q_n) \in \mathbb Z^n$, with $\mathbf q \ge 1$ (where we mean that $q_i \ge 1$ for all $i$), write $\mathbb U_{\mathbf q} = \mathbb U_{q_1} \times \dots \times \mathbb U_{q_n}$; also we set $\mathbf{a/q} = (a_1/q_1, \dots, a_n/q_n)$ and $\bfa\bfq = (a_1q_1, \dots, a_nq_n)$ if $\mathbf a = (a_1, \dots, a_n)$ is another vector. For $\lambda \in \mathbb Z$ and $\mathbf a, \mathbf q \in \mathbb Z^n$, with $\mathbf q \ge 1$, define
\begin{gather*} 
g(a,q; b,r) = \frac 1{\varphi([q,r])} \sum_{x \in \mathbb U_{[q,r]}} e\bigg( \frac {ax^k}q + \frac {bx}r \bigg), \\
\singseries = \sum_{q=1}^\infty \sum_{a \in \mathbb U_q} e( -\lambda a/q ) \prod_{i=1}^\dimension g(a, q; a_i, q_i),
\end{gather*}
where $\varphi$ is Euler's totient function and $[q,r] = \lcm[q,r]$. Choose a smooth bump function $\psi$ such that
\[ \mathbf 1_{\mathcal Q}(\mathbf x) \le \psi(\mathbf x) \le \mathbf 1_{\mathcal Q}(\mathbf x/2), \]
where $\mathbf 1_{\mathcal Q}$ is the indicator function of the cube $\mathcal Q = [-1,1]^n$.  Finally,  write $\psi_t(\mathbf x) = \psi(t\mathbf x)$.

\begin{thm}[From \cite{ACHK}] \label{thm:approximation_formula}
Let $k \geq 2$ and $n$ be large enough. Also, let $\lambda \in \Gammakn$ be large, and suppose that $\lambda^{1/k} \le N \lesssim \lambda^{1/k}$.  For any fixed $B>0$, there exists a $C = C(B) > 0$ such that one has the decomposition 
\begin{equation*}
\widehat{\sigma_\lambda} (\bm\xi) = \frac {\lambda^{n/k-1}}{P(\lambda)} \sum_{ 1 \le \bfq \le Q} \sum_{\bfa \in \mathbb U_{\bfq}} \singseries \psi_{N/Q}(\bfq\bm\xi-\bfa) \widetilde{d\sigma_{\lambda}}(\bm\xi-\bfa/\bfq) + \widehat {E_\lambda}(\bm\xi) := \widehat {M_\lambda}(\bm\xi)+  \widehat {E_\lambda}(\bm\xi),
\end{equation*}
where $Q = (\log N)^C$ and $\widetilde{d\sigma_{\lambda}}$ is the Fourier transform of the continuous $k$-spherical surface measure.
\end{thm}

We recall here that this Theorem was stated for slightly different values of $n$.  However, that statement included an $l^2(\Z^n)$ bound for a certain dyadic operator that we do not need here.  Therefore, Theorem \ref{thm:approximation_formula} as stated here, and as used later on, actually holds for the larger range $n \geq n_0(k)$.

Specifically, we prove the following quantitative $l^p(\Z^n)$-improving decay rate:

\begin{thm}
\label{maintheorem}
Let $n \geq n_0(k)$ and the spherical averages $A_\lambda$ be defined as above.  Then we have the following quantitative $l^p(\Z^n)$ improving inequality, with decay in $\lambda$:\\
For $k \geq 3$ and $1<p<2$
\begin{equation*}
   \|A_\lambda f\|_{l^{p'}(\Z^n)} \leq C_{B,p,n,k}\lambda^{(1-\frac{n}{k})(\frac{2}{p}-1)}(\log \lambda)^{-B}\|f\|_{l^p(\Z^n)}
\end{equation*}
and for $k=2$ and $1<p<2$
\begin{equation*}
   \|A_\lambda f\|_{l^{p'}(\Z^n)} \leq C_{p,n}\lambda^{(1-\frac{n}{2})(\frac{2}{p}-1)}(\log \lambda)^{n}\|f\|_{l^p(\Z^n)}
\end{equation*}
\end{thm}

The power decay comes from the trivial bound; for $k\geq 3$ additionally we gain a logarithmic decay. No such gain is present for $k=2$ due to the fact that the main term is larger than the error term, see the remarks after the proof of Theorem \ref{maintheorem}.  This logarithmic decay is likely the best possible using these methods; additional decay might be possible due to increased knowledge about the distribution of primes and would likely require the resolution of deep problems in analytic number theory.

We can also say something similar for the discrete dyadic prime spherical maximal function: see Theorem \ref{dyadic}.  Both of these theorems follow by a straightforward interpolation argument; however, the averages for the main term of the decomposition, $M_\lambda$, satisfy a better quantitative improving bound for $k \geq 3$, whose proof requires a careful decomposition of the corresponding multiplier, which is where the main work of this paper lies.
\begin{thm}
\label{mainterm}
Let $n$ be as in Theorem \ref{maintheorem}.  Then the averaging operator $M_\lambda$ (see Theorem \ref{thm:approximation_formula}) satisfies the following improving estimate for $k \geq 3$ and $1<p<2$:
\begin{equation*}
   \|M_\lambda f\|_{l^{p'}(\Z^n)} \leq C_{p,n,k}\lambda^{(\frac{2-n}{k})(\frac{2}{p}-1)}(\log N)^{C(\frac{2}{p}-1)}\}\|f\|_{l^p(\Z^n)}.
\end{equation*}
\end{thm}

The main term operator contains both the arithmetic and analytic content of the spherical averages, see e.g. \cite{ACHK}, \cite{MSW} for a discussion of this.  Also, the maximal function of the main term satisfy $l^p$ bounds for all $p> \frac{n}{n-2}$, independent of the degree $k$ (as long as the degree is large enough).  Proving this non-trivial $l^p$-improving of these main term operators may lead to insight on how to improve the error term in the decomposition, thus leading to improved $l^p$ bounds for the spherical maximal operator.  Additionally, knowledge of $l^p$-improving for the main term will help to compare on a structural level the similarities and differences between the prime and integer spherical maximal functions; for example, might the optimal dependence on $\lambda$ be independent of the degree $k$, as in the $l^p \to l^p$ bound?  Finally, better knowledge of quantitative $l^p$ improving estimates (for the main term or the entire operator) would likely lead to better quantitative sparse bounds, recently pursued in the integer setting \cite{K2}, but not yet for this prime variant. 

Most of the next and final section of this paper is devoted to proving Theorem \ref{mainterm}.  We show a certain decay rate for the main term $M_\lambda$ of the operator, which requires a careful analysis of its corresponding multiplier.  This immediately leads to Theorem \ref{mainterm}.

We use a boldface script to denote multidimensional vectors in dimension $n$, where the underlying space they belong to ($\R^n$, $\Z^n$, or $\T^n$) will be clear from the context.  Moreover the notation $A\lesssim B$ will be used to mean $ A \leq CB$ where $C$ is a constant that may depend on $p$, $n$, or $k$, but never on $\lambda$.
The notation used in this paper will be introduced as needed.  The next section contains all the proofs and a brief discussion.

\subsection{Acknowledgements}
The author was supported by NSF grant DMS-1502464.  She would like to thank Kevin Hughes for suggesting the extension to the dyadic maximal function, and the anonymous referee for expository suggestions and for catching an error in the original manuscript.

\section{Proofs and discussion}
  Theorem \ref{maintheorem} is proved by interpolating the following estimates:
\begin{equation}
\label{1}
    \|A_\lambda f\|_{l^\infty(\Z^n)} \lesssim \lambda^{1-\frac{n}{k}}(\log{\lambda}^{1/k})^n\|f\|_{l^1(\Z^n)}
\end{equation}
\begin{equation}
\label{2}
     \|M_\lambda f\|_{l^\infty(\Z^n)} \lesssim \lambda^{\frac{2-n}{k}}(\log{N})^C\|f\|_{l^1(\Z^n)}
\end{equation}
\begin{equation}
\label{3}
     \|A_\lambda f\|_{l^2(\Z^n)} \lesssim \|f\|_{l^2(\Z^n)}
\end{equation}
\begin{equation}
\label{4}
     \|E_\lambda f\|_{l^2(\Z^n)} \lesssim (\log{N})^{-B}\|f\|_{l^2(\Z^n)}
\end{equation}

Given these estimates, we prove Theorem \ref{maintheorem}:
\begin{proof}[Proof of Theorem \ref{maintheorem}]
Let $k \geq 3$ and $1 \leq p \leq 2$.  First, \eqref{1} gives
\[
     \|E_\lambda f\|_{l^\infty(\Z^n)} \lesssim \lambda^{1-\frac{n}{k}}(\log{\lambda}^{1/k})^n\|f\|_{l^1(\Z^n)}
\]
since $\frac{2-n}{k} < 1-\frac{n}{k}<0$.  Interpolating this with $ \|M_\lambda f\|_{l^2(\Z^n)} \lesssim \|f\|_{l^2(\Z^n)}$ derived from \eqref{4} we get
\[
     \|E_\lambda f\|_{l^{p'}(\Z^n)} \lesssim \lambda^{(1-\frac{n}{k})(\frac{2}{p}-1)}(\log{\lambda}^{1/k})^{n(\frac{2}{p}-1)}(\log N)^{-B(2-\frac{2}{p})}\|f\|_{l^p(\Z^n)}.
\]
Note that we can chose $B \geq n$ which simplifies the expression.
For the main term, interpolating \eqref{2} and \eqref{3} gives
\[
     \|M_\lambda f\|_{l^{p'}(\Z^n)} \lesssim \lambda^{(\frac{2-n}{k})(\frac{2}{p}-1)}(\log N)^{C(\frac{2}{p}-1)}\|f\|_{l^p(\Z^n)}.
\]
Putting these together, we get
\[
     \|A_\lambda f\|_{l^{p'}(\Z^n)} \lesssim max\{ \lambda^{(1-\frac{n}{k})(\frac{2}{p}-1)}(\log N)^{-B} , \lambda^{(\frac{2-n}{k})(\frac{2}{p}-1)}(\log N)^{C(\frac{2}{p}-1)}\}\|f\|_{l^p(\Z^n)}
\]

which gives Theorem \ref{maintheorem}.

For $k=2$, we can do better by simply interpolating \eqref{1} and \eqref{3}, which gives the trivial interpolation bound of 
\[
\|A_\lambda f\|_{l^{p'}(\Z^n)} \lesssim \lambda ^{(1-\frac{n}{k})(\frac{2}{p}-1)}(\log \lambda)^{n}\|f\|_{l^p(\Z^n)}
\]
\end{proof}

\begin{proof}[Proof of Theorem \ref{mainterm}]
Note that the second term in the maximum that appears in the last inequality of the previous proof comes from the main term operator $M_{\lambda}$.  Therefore, once estimates \eqref{2} is proved, we get Theorem \ref{mainterm}.
\end{proof}

\begin{remark}
The decay rate of $\lambda$ improves as $p$ approaches 1.  At $p=2$, we get no improvement.  This makes sense since we expect better decay at lower values of $p$ since improving is ``trivial" in this discrete setting.
\end{remark}

\begin{remark}
Note that these estimates hold for $n\geq 5$ if $k=2$.  This is in contrast to some of the results in \cite{ACHK} which only hold for $n \geq 7$.   This might provide further evidence that the spherical maximal function along the primes might be bounded for all $n\geq 5$.
\end{remark}

\begin{remark}
For $k \geq 3$, we have that $\frac{2-n}{k} < 1-\frac{n}{k}$.  This is due to the fact that the decay from the main term is greater for $p<2$, as is expected.  Any improvements to the estimate \eqref{4} would automatically improve Theorem \ref{maintheorem}.  For the case $k=2$, this decomposition from Theorem \ref{thm:approximation_formula}  gives no improvement to the interpolation between the easy estimates \eqref{1} and \eqref{3}, because in this case the error term actually has better decay.  
\end{remark}

\begin{remark}
The decay in Theorem \ref{mainterm} is likely not optimal.  To discuss optimality, we will likely have to restrict to a smaller range of $p$, such as in \cite{KL}.  Any improvements to estimate \eqref{2} would directly improve this decay rate.  It seems a plausible conjecture that once we restrict to a certain range of $p$, the optimal decay rate is $   \|M_\lambda f\|_{l^{p'}(\Z^n)} \lesssim \lambda^{-\frac{n}{k}(\frac{2}{p}-1)}(\log \lambda)^{C(\frac{2}{p}-1)}\|f\|_{l^p(\Z^n)}$.  However, the corresponding conjecture in the integer case is not fully known in terms of the range of $p$.  Any improvements to estimate \eqref{2} would directly improve this decay rate in our case, but it is likely that improvements to Theorem \ref{mainterm} will come from other techniques. 
\end{remark}

We now prove \eqref{1} through \eqref{4}.
We use the trivial estimate to get \eqref{1}: due to the Hua asymptotic for $P(\lambda)$, we have
\[
|A_\lambda f(\mathbf l)| \leq  \frac{1}{\lambda^{\frac{n}{k}-1}} \sum_{|{\bf p}|^k=\lambda} (\log{\bf p}) f({\bf p-l}), 
\]
which can be bounded trivially in $l^\infty(\Z^n)$ norm by $\lambda^{1-\frac{n}{k}}(\log{\lambda}^{1/k})^n\|f\|_{l^1(\Z^n)}$.  Also, 
\eqref{3} is also easily seen to be true since $\sigma_\lambda$ is defined to be a probability measure.

We now describe how to get \eqref{4}.  Note that from \cite{ACHK} we have that 
\begin{equation}
\left\| \widehat{E_\lambda} \right\|_{L^\infty(\mathbb T^n)} \lesssim (\log \la)^{-B}. 
\end{equation}
and moreover, that this actually holds for a larger range, including all $n \geq 5$ when $k=2$.  With this in mind, using Plancherel's theorem and properties of the Fourier transform, we have
\[
\|E_\lambda f\|_{l^2(\Z^n)}
 = \|\widehat{E_\lambda f}\|_{l^2(\Z^n)} = \|\widehat{E_\lambda}\widehat{f}\|_{l^2(\Z^n)} \leq \|\widehat{E_\lambda}\|_{L^\infty(\T^n)}\|f\|_{l^2(\Z^n)} \lesssim 
 (\log \la)^{-B}\|f\|_{l^2(\Z^n)}
\]
which gives \eqref{4} as claimed.

The remainder of the paper is devoted to proving \eqref{2}.  We follow the method outlined in \cite{Hughes} with the technology and notation from \cite{ACHK}; we prove a favorable $L^\infty$ estimate for the kernel of the main term operator, which is the inverse Fourier transform of $\widehat{M_\lambda}$.  This is because
\[
\|M_\lambda f\|_{l^\infty(\Z^n)} = \|K_\lambda \star f\|_{l^\infty(\Z^n)} \lesssim \|K_\lambda \|_{l^\infty(\Z^n)}\|f\|_{l^1(\Z^n)}
\]
by Young's inequality (where $K_\lambda$ is the kernel of the convolution operator), so if we can show

\begin{equation}
\label{kernel decay}
    \|K_\lambda \|_{l^\infty(\Z^n)} \lesssim \lambda^{\frac{2-n}{k}}(\log N)^C
\end{equation}
then we will get \eqref{2}.

Recall that we can maneuver the sums in the variable $a$ and $q$ as well as the multidimensional sums in $\bfa$ and $\bfq$, and note that even though the sum in $q$ in \cite{ACHK} extends to $\infty$, this was chosen purely for convenience, and that this sum needs only to be taken to $N$.  
With this in mind, denote 
\[
G_\lambda (\bfa, \bfq) = \prod_{i=1}^\dimension g(a, q; a_i, q_i)
\]
 and 
 \[
\psi_{N/Q, \bfq}(\bm\xi-\bfa/\bfq) =  \psi_{N/Q}(\bfq\bm\xi-\bfa)
\]
so that 
\[\widehat{K_\lambda} := \widehat{M_\lambda}(\bm \xi) = \sum_{q=1}^N \sum_{a\in \mathbb U_q} \widehat{K_\lambda^{a,q}}(\bm \xi), \]
where 
\[ \widehat{K_\lambda^{a,q}}(\bm\xi) :=   e\left( -\lambda a/q \right)\sum_{\bfq \leq Q}\sum_{\mathbf a \in \mathbb U_{\bfq}} G_\lambda (\bfa, \bfq) \psi_{N/Q, \bfq}(\bm\xi-\bfa/\bfq) \widetilde{d \sigma_{\lambda}}(\bm\xi-\mathbf a/\mathbf q). \]

 To show \eqref{kernel decay}, we start by applying Fourier inversion
 \[
   K_\lambda^{a,q}({\bf x}) = e(-\lambda \cdot a/q)\int_{\T^n}e(- {\bf x}\cdot \bm\xi)\sum_{ 1 \le \bfq \le Q}\sum_{\bfa \in \mathbb U_{\bfq}} G_\lambda (\bfa ,\bfq)\psi_{N/Q, \bfq}(\bm\xi-\bfa/\bfq) \widetilde{d\sigma_{\lambda}}(\bm\xi-\bfa/\bfq) 
   \]
   and noting that for each fixed $\xi$, the supports of the $\psi_{N/Q}(\bfq\bm\xi-\bfa)$ are disjoint in $\bfq$ (for $\bfq \leq Q$), we get
   \[
      K_\lambda^{a,q}({\bf x}) = e(-\lambda \cdot a/q)\int_{\T^n}e(- {\bf x}\cdot \bm\xi)\sum_{\bfa \in \mathbb U_{\bfq}} G_\lambda(\bfa , \bfq)\psi_{N/Q, \bfq}(\bm\xi-\bfa/\bfq) \widetilde{d\sigma_{\lambda}}(\bm\xi-\bfa/\bfq).
    \]
  Writing out the Gauss sum along with multiplying and dividing by $e(\frac{a_1x_i}{q_i})$, the integral becomes
    \[
\int_{\T^n}e(- {\bf x}\cdot \bm\xi)\sum_{\bfa \in \mathbb U_{\bfq}} \prod_{i=1}^\dimension \frac 1{\varphi([q,q_i])} \sum_{b \in \mathbb U_{[q,q_i]}} e\bigg( \frac {ab^k}q + \frac {a_ib}{q_i} \bigg)e(\frac{a_1x_i}{q_i})e(\frac{-a_1x_i}{q_i})\psi_{N\/Q, \bfq}(\bm\xi-\bfa/\bfq) \widetilde{d\sigma_{\lambda}}(\bm\xi-\bfa/\bfq).
\]
Since we are integrating over $\bm\xi$, due to the cutoff function $\psi$ we can extend the integration to $\R^n$.  We therefore get 
\[
    K_\lambda^{a,q}({\bf x}) = e(-\lambda \cdot a/q)
G_x(a,q,\bfq)\int_{\R^n}e(-{\bf x}(\bm\xi - \bfa/\bfq))\psi_{N/Q, \bfq}(\bm\xi-\bfa/\bfq) \widetilde{d\sigma_{\lambda}}(\bm\xi-\bfa/\bfq) 
\]

where 
\begin{equation}
\label{Gausscomponent}
   G_x(a,q,\bfq) = \sum_{\bfa \in \mathbb U_{\bfq}} \prod_{i=1}^\dimension \frac 1{\varphi([q,q_i])} \sum_{b \in \mathbb U_{[q,q_i]}} e\bigg( \frac {ab^k}q + \frac {a_ib}{q_i} \bigg)e\bigg(-\frac{a_1x_i}{q_i}\bigg) 
\end{equation}

since we have
\[
e(-\bm x\cdot \bm\xi)\prod_{i=1}^n e(a_ix_i/q_i) = e(-\sum_{i=1}^n x_i(\xi_i-\frac{a_i}{q_i})) = e(-\bm x(\bm\xi-\bfa/\bfq)).
\]
Putting this all together, we get
\[
 K_\lambda^{a,q}({\bf x}) = e(-\lambda \cdot a/q)
G_x(a,q,\bfq)
\mathcal{F}(\psi_{N/Q, \bfq}(\bm\xi-\bfa/\bfq) \widetilde{d\sigma_{\lambda}}(\bm\xi-\bfa/\bfq))
\]
where $\mathcal{F}$ indicates the Fourier transform. 

Using properties of the Fourier transform, we can rewrite
\[
\mathcal{F}(\psi_{N/Q, \bfq}(\bm\xi-\bfa/\bfq) \widetilde{d\sigma_{\lambda}}(\bm\xi-\bfa/\bfq)) = 
\widetilde{\psi_{N/Q, \bfq}}\star d\sigma_{\lambda}(\bm x).
\]
to get 
\begin{equation}
\label{Kerneldecomposition}
K_\lambda^{a,q}({\bf x}) = e(-\lambda \cdot a/q)G_x(a,q,\bfq)\widetilde{\psi_{N/Q, \bfq}}\star d\sigma_{\lambda}(\bm x)
\end{equation}
so that 
\[
 |K_\lambda^{a,q}({\bf x})| =
|G_x(a,q,\bfq)|
|\widetilde{\psi_{N/Q, \bfq}}\star d\sigma_{\lambda}(\bm x)|.
\]
Next we will bound both the Gauss component \eqref{Gausscomponent} and the convolution in absolute values.
We can rewrite \eqref{Gausscomponent} as
\[
G_x(a,q,\bfq) =  \prod_{i=1}^\dimension\sum_{a_i\in U_{q_i}}\frac 1{\varphi([q,q_i])} \sum_{b \in \mathbb U_{[q,q_i]}} e\bigg( \frac {ab^k}q + \frac {a_ib}{q_i} \bigg)e\bigg(\frac{-a_1x_i}{q_i}\bigg).
\]
\[
=  \prod_{i=1}^\dimension\sum_{a_i\in U_{q_i}}g(a, q; a_i, q_i)e\bigg(\frac{-a_1x_i}{q_i}\bigg)
\]
(To see this note that if $a_i\in U_{q_i}$ then there exists a $n_i$ such that $n_ia_i = 1$, so $\lcm_i(n_i)a_i = 1$, so $\bfa \in U_\bfq$.  Conversely, $\bfa\in U_\bfq$ implies $a_i\in U_{q_i}$ for all $i$.)


We have the following bound:
\begin{lemma}
For all $\varepsilon >0$,
\[
 |\sum_{a_i\in U_{q_i}}g(a, q; a_i, q_i)e\bigg(\frac{-a_1x_i}{q_i}\bigg)| \lesssim q_i^{\varepsilon}
 \]
\end{lemma}
\begin{proof}
From the proof of Lemma 6, part (iii) in \cite{ACHK}, we have that 
\[
|\sum_{a_i\in U_{q_i}}g(a, q; a_i, q_i)e\bigg(\frac{-a_1x_i}{q_i}\bigg)| \leq \frac{\tau((q,q_i))}{\varphi(q_i/(q,q_i))}\sum_{d|\frac{q_i}{(q,q_i)}}d = \frac{\tau((q,q_i))}{\varphi(q_i/(q,q_i))}\sigma(q_i/(q,q_i))
\]
Now recall the following facts:
\begin{itemize}
    \item $\tau(n) \lesssim n^{\varepsilon '}$
    \item $n^{1-\varepsilon ''} \leq \varphi(n) \leq n$
    \item $\sigma(n) \leq e^\gamma n\log\log n$
\end{itemize}
for all $\varepsilon', \varepsilon'' >0$.  Here $\tau(n) = \sum_{d|n}1$ and $\sigma(n) = \sum_{d|n}d$, while $\gamma$ is the Euler-Mascheroni constant.

Using these facts, we get the bound
\[
e^\gamma (q,q_i)^{\varepsilon '} \bigg(\frac{q_i}{(q,q_i)}\bigg)^{\varepsilon ''}\log\log \bigg(\frac{q_i}{(q,q_i)}\bigg) \lesssim q_i^{\varepsilon}
\]
\end{proof}

Therefore
\[
|G_x(a,q,\bfq)| \leq \prod_{i=1}^\dimension q_i^\varepsilon
\]

To bound the convolution in \eqref{Kerneldecomposition}, we use a variant of the well-known decay of the spherical measure (where we have the diagonal transformation $\bfq$), we have that 
\[
\widetilde{\psi_{N/Q, \bfq}}\star d\sigma_{\lambda}(\bm x) \lesssim \frac{QN^{-n}}{\bfq(1+\frac{|\bm x|}{N\bfq})^M}
\]
where $M>0$ is any natural number and the implicit constant depends on $M$ (see, for example, equation (5.5.12) in \cite{grafakos08a} - this also holds for degree $k$ spheres).

Now we show \eqref{kernel decay}.   Trivially summing over $a\in U_q$, we have
\begin{equation}
\label{sum in a}
|\sum_{a\in U_q}K_\lambda^{a,q}(\bm x)| \leq q|K_\lambda^{a,q}(\bm x)| \leq q \prod_{i=1}^\dimension q_i^\varepsilon|\widetilde{\psi_{N\/Q, \bfq}}\star d\sigma_{\lambda}(\bm x)| \lesssim qQ^{1+\varepsilon}N^{-n}
\end{equation}

and we have (by technically abusing notation in that the constant $C$ below should be $C' = C(1+\varepsilon)$),
\[
|\sum_{q=1}^N\sum_{a\in U_q}K_\lambda^{a,q}(\bm x)| \lesssim N^2Q^{1+\varepsilon}N^{-n} = \lambda ^{\frac{2-n}{k}}Q^{1+\varepsilon} = \lambda ^{\frac{2-n}{k}}(\log N)^C
\]
which is \eqref{kernel decay}.

Finally, we indicate how to prove the following maximal dyadic version of Theorem \ref{maintheorem}:

\begin{thm}
\label{dyadic}
Let $n \geq n_1(k)$ (see \cite{ACHK} for a precise definition).  Then we have for $k \geq 3$,, $1<p<2$
\begin{equation*}
   \left\|\sup_{\Lambda \leq \lambda < 2\Lambda}|A_\lambda f|\right\|_{l^{p'}(\Z^n)} \lesssim C_{B,p,n,k}\Lambda^{(1-\frac{n}{k})(\frac{2}{p}-1)}(\log \Lambda)^{-B}\|f\|_{l^p(\Z^n)}
\end{equation*}
and for $k=2$ and $1<p<2$
\begin{equation*}
   \|A_\lambda f\|_{l^{p'}(\Z^n)} \leq C_{p,n}\Lambda^{(1-\frac{n}{2})(\frac{2}{p}-1)}(\log \Lambda)^{n}\|f\|_{l^p(\Z^n)}
\end{equation*}
\end{thm}
\begin{proof}
We can show the analogues of \eqref{1} through \eqref{4} and imterpolate as in the proof of Theorem \ref{maintheorem}.  The only estimates that require commentary are proving the analogues of \eqref{2} and \eqref{4}.  Instead of \eqref{4}, we need to show
\begin{equation}
     \left\| \sup_{\Lambda \leq \lambda < 2\Lambda}|E_\lambda f|\right\|_{l^2(\Z^n)} \lesssim (\log{\Lambda})^{-B}\|f\|_{l^2(\Z^n)},
\end{equation} but this is the content of Theorem 1 of \cite{ACHK}, where we need the stronger assumption that $n\geq n_1(k)$ (namely $n \geq 7$ for $k=2$).  Finally we show 
\begin{equation}
     \left\| \sup_{\Lambda \leq \lambda < 2\Lambda}|M_\lambda f|\right\|_{l^\infty(\Z^n)} \lesssim \Lambda^{\frac{2-n}{k}}(\log{N})^C\|f\|_{l^1(\Z^n)}
\end{equation}
 which proceeds by a very similar argument as \eqref{2}.  One simply uses the fact that 
 \[
 \left\|\sup_{\Lambda \leq \lambda < 2\Lambda}|K_\lambda\star f|\right\|_{l^\infty(\Z^n)} = esssup_{\bm{x}}|\sup_{\Lambda \leq \lambda < 2\Lambda}|K_\lambda\star f|| =  \sup_{\Lambda \leq \lambda < 2\Lambda}esssup_{\bm{x}}|K_\lambda\star f|
 \]
 \[
 \leq \sup_{\Lambda \leq \lambda < 2\Lambda}\|K_\lambda\|_{l^\infty(\Z^n)}\|f\|_{l^1(\Z^n)}.
 \]

\end{proof}
\bibliographystyle{amsplain}

\begin{thebibliography}{99}

\bibitem{ACHK}
T. Anderson, B. Cook, K. Hughes, and A. Kumchev, \emph{On the Ergodic Waring-Goldbach Problem}. Submitted.  Preprint on arXiv: arXiv:1703.02713.

\bibitem{grafakos08a}
L.~Grafakos, \emph{Classical {F}ourier {A}nalysis}, Second ed., Springer, 2008.

\bibitem{Hua_book}
L.~K. Hua, \emph{Additive {T}heory of {P}rime {N}umbers}, American Mathematical
  Society, 1965.

\bibitem{Hughes}
K. Hughes, \emph{$\ell^p$-improving for discrete spherical averages}.  Preprint available on arXiv: arXiv:1804.09260.


\bibitem{K1}
R. Kesler, \emph{	
$\ell^p(\mathbb{Z}^d)$-Improving Properties and Sparse Bounds for Discrete Spherical Maximal Averages}.  Preprint available on arXiv: arXiv:1805.09925.

\bibitem{K2}
R. Kesler, \emph{$\ell^p(\mathbb{Z}^d)$-Improving Properties and Sparse Bounds for Discrete Spherical Maximal Means, Revisited}.   Preprint available on arXiv: arXiv:1809.06468.

\bibitem{KL}
R. Kesler and M Lacey, \emph{
$ \ell ^{p}$-improving inequalities for Discrete Spherical Averages}.  Preprint available on arXiv: arXiv:1804.09845.

\bibitem{KLM}
R. Kesler, M. Lacey and D. Mena. \emph{
Lacunary Discrete Spherical Maximal Functions}.  Preprint available on arXiv:  arXiv:1810.12344.

\bibitem{Littman}
W. Littman.  \emph{$L^p-L^q$-estimates for singular integral operators arising from hyperbolic equations.} Partial differential equations (Proc. Sympos. Pure Math., Vol. XXIII, Univ. California, Berkeley, Calif., 1971), pp. 479–481

\bibitem{Magyar discrete}
A.~Magyar, \emph{{$L^p$}-bounds for spherical maximal operators on {$\mathbb
  Z^n$}}, Rev. Mat. Iberoamericana \textbf{13} (1997), no.~2, 307--317.

\bibitem{MSW}
A.~Magyar, E.~M. Stein, and S.~Wainger, \emph{Discrete analogues in harmonic
  analysis: {S}pherical averages}, Ann. of Math. (2) \textbf{155} (2002),
  no.~1, 189--208. 

\bibitem{Stein93}
E.~M. Stein, \emph{Harmonic Analysis: Real-Variable Methods, Orthogonality, and Oscillatory Integrals}, Princeton University Press, 1993.

\bibitem{Str}
R. Strichartz. \emph{Convolutions with kernels having singularities on a sphere.} Trans. Amer. Math. Soc. 148 (1970), 461-471

\end{thebibliography}

\end{document}